\documentclass[12pt]{amsart}
\usepackage{latexsym,enumerate}
\usepackage{amssymb, xcolor,mathrsfs}
\usepackage{amsmath,amsthm,amsfonts,amssymb,latexsym, mathabx}
\usepackage[pagebackref,colorlinks]{hyperref}
\usepackage{arydshln}

\newtheorem{theorem}{Theorem}[section]
\newtheorem{lemma}[theorem]{Lemma}
\newtheorem{proposition}[theorem]{Proposition}

\newtheorem{theorema}{Theorem}

\theoremstyle{definition}

\newtheorem{remark}[theorem]{Remark}

\numberwithin{equation}{section}

\newcommand{\GL}{{\mathrm {GL}}}
\newcommand{\PGL}{{\mathrm {PGL}}}
\newcommand{\SL}{{\mathrm {SL}}}
\newcommand{\PSL}{{\mathrm {PSL}}}

\newcommand{\GU}{{\mathrm {GU}}}

\newcommand{\SU}{{\mathrm {SU}}}
\newcommand{\PSU}{{\mathrm {PSU}}}

\newcommand{\Ker}{\operatorname{Ker}}

\newcommand{\Irr}{{\mathrm {Irr}}}

\newcommand{\ord}{{\mathrm {ord}}}

\renewcommand{\Im}{{\mathrm {Im}}}

\newcommand{\diag}{{\mathrm {diag}}}

\newcommand{\Syl}{{\mathrm {Syl}}}
\newcommand{\St}{{\mathsf {St}}}

\newcommand{\lcm}{{\mathrm {lcm}}}

\newcommand{\QQ}{{\mathbb Q}}
\newcommand{\ZZ}{{\mathbb Z}}

\newcommand{\FF}{{\mathbb F}}

\newcommand{\EC}{\mathcal{E}}

\newcommand{\bC}{{\mathbf{C}}}
\newcommand{\bG}{{\mathbf{G}}}
\newcommand{\bS}{{\mathbf{S}}}

\newcommand{\bT}{{\mathbf{T}}}

\newcommand{\bB}{{\mathbf{B}}}
\newcommand{\bU}{{\mathbf{U}}}

\newcommand{\bZ}{{\mathbf{Z}}}

\newcommand{\Al}{\textup{\textsf{A}}}
\newcommand{\Sy}{\textup{\textsf{S}}}

\makeatletter \@namedef{subjclassname@2020}{\textup{2020}
Mathematics Subject Classification} \makeatother

\def\nor{\trianglelefteq\,}

\begin{document}

\title[Hall $\pi$-subgroups and characters of $\pi'$-degree]
{Hall $\pi$-subgroups and characters of $\pi'$-degree}

\author[Eugenio Giannelli]{Eugenio Giannelli}
\address{Dipartimento di Matematica e Informatica U. Dini, Firenze,
Italy}
\email{eugenio.giannelli@unifi.it}

\author[Nguyen N. Hung]{Nguyen N. Hung}
\address{Department of Mathematics, The University of Akron, Akron,
OH 44325, USA}
\email{hungnguyen@uakron.edu}

\thanks{The first author's research is funded by: the European Union Next
Generation EU, M4C1, CUP B53D23009410006, PRIN 2022 - 2022PSTWLB
Group Theory and Applications; and the INDAM-GNSAGA Project CUP
E53C24001950001. The second author gratefully acknowledges support
from the AMS--Simons Research Enhancement Grant (AWD-000167 AMS). We
thank Gunter Malle for several helpful comments on an earlier
version, particularly regarding
Proposition~\ref{prop:linear-unitary}}



\thanks{}

\begin{abstract}
We study the relationship between the existence of Hall
$\pi$-subgroups and that of irreducible characters of $\pi'$-degree
with prescribed fields of values in finite groups. This work extends
a result of Navarro and Tiep from a single odd prime to multiple odd
primes.
\end{abstract}

\maketitle



\section{Introduction}

A well-known conjecture of R.~Gow, proved by G.~Navarro and
P.~H.~Tiep~\cite{Navarro-Tiep2}, asserts that if $G$ is a finite
group of even order, then $G$ possesses a nontrivial (complex)
irreducible character of odd degree whose values are rational. This
result was subsequently generalized to all primes by the same
authors, as follows.

\begin{theorem}[\cite{Navarro-Tiep1}, Theorem~A]\label{thm:Navarro-Tiep}
Let $G$ be a finite group of order divisible by a prime $p$. Then
$G$ has a nontrivial irreducible character of degree not divisible
by $p$ whose values lie in $\QQ(e^{2\pi i/p})$.
\end{theorem}

In~\cite{GHSV21}, together with A.\,A.~Schaeffer Fry and C.~Vallejo,
we attempted to extend this theorem from a single prime $p$ to a set
$\pi$ of two primes. Among other results, it was shown that if
$\pi=\{2,p\}$ and $\gcd(|G|,2p)>1$, then $G$ possesses a nontrivial
irreducible character of $\pi'$-degree whose values are contained in
$\QQ(e^{2\pi i/p})$. (Here, a character $\chi$ is said to have
$\pi'$-degree if $\chi(1)$ is not divisible by any prime in $\pi$.)
Unfortunately, this phenomenon does not hold in general when $\pi$
consists of two odd primes. For example, as pointed out in
\cite[Proposition~4.1]{GHSV21}, the Tits group ${}^2F_4(2)'$ has no
nontrivial irreducible character of $\{3,5\}'$-degree with values in
$\QQ(e^{2\pi i/15})$.

In this paper, we propose a different extension -- perhaps a more
natural one -- of Theorem~\ref{thm:Navarro-Tiep} that works for
an arbitrary set of odd primes. In the following, we write $\pi(G)$
for the set of prime divisors of $|G|$. Recall that a subgroup $H\le
G$ is called a Hall $\pi$-subgroup of $G$ if $\pi(H)\subseteq \pi$
and $|G:H|$ is not divisible by any prime in $\pi$.

\begin{theorema}\label{thm:main}
Let $\pi$ be a set of odd primes and $G$ a finite group possessing a
nontrivial Hall $\pi$-subgroup. Then $G$ has a nontrivial
$\pi'$-degree irreducible character with values in $\QQ(e^{2\pi
i/p})$ for some $p\in\pi$.
\end{theorema}

Theorem~\ref{thm:main} may be viewed as a generalization of
Theorem~\ref{thm:Navarro-Tiep}. In fact, when $p$ is odd,
Theorem~\ref{thm:Navarro-Tiep} follows from the case $\pi=\{p\}$ of
Theorem~\ref{thm:main} together with the first Sylow theorem.

If we assume that $2\in \pi$, then the conclusion of
Theorem~\ref{thm:main} remains true when $|\pi|\leq 2$, as noted
above, but fails in general once $|\pi|\geq 3$. For example, if $G$
is any non-abelian simple group and $\pi=\pi(G)$, then the trivial
character $1_G$ is the only irreducible character of $G$ with
$\pi'$-degree. There are also counterexamples with $\pi\varsubsetneq
\pi(G)$, such as $(G,\pi)=(\Al_7,\{2,3,5\})$. At present, we do not
have a conceptual explanation for why the prime $2$ is special in
this context.

Similar to Gow's conjecture and the results in
\cite{Navarro-Tiep1,GHSV21}, Theorem~\ref{thm:main} admits a clean
reduction to finite simple groups. Accordingly, the main body of
this paper is devoted to proving the result for such groups. This
requires some new work on the relationship between the existence of
Hall $\pi$-subgroups and $\pi'$-degree irreducible characters in
simple groups of Lie type.

\section{Non-Abelian Simple Groups}

This section is devoted to proving Theorem~\ref{thm:main} for all
non-abelian simple groups. We first fix some notation. For a
positive integer $k$, we let $\zeta_k:=e^{2\pi i/k}$ and we write $\QQ(\zeta_k)$ for the $k$th
cyclotomic field. As usual, $\Irr(G)$ denotes
the set of all irreducible characters of a group $G$, and
$\Irr_{\pi'}(G)$ denotes the subset consisting of those characters
whose degrees are not divisible by any prime in $\pi$. For
$\chi\in\Irr(G)$, we write $\QQ(\chi)$ for the field of values of
$\chi$, that is, the smallest extension of $\QQ$ containing all
values of $\chi$. Finally, for integers $x\le y$, we let
$[x,y]:=\{\,z\in\ZZ \mid x\le z\le y\,\}$.

As mentioned above, we aim at proving the following.

\begin{theorem}\label{thm:simple}
Let $S$ be a non-abelian simple group and $\pi$ be a set of odd
primes such that $S$ has a Hall $\pi$-subgroup. Then there exists
$1_S\neq\chi\in\Irr_{\pi'}(S)$ such that
$\QQ(\chi)\subseteq\QQ(\zeta_p)$ for some $p\in\pi$.
\end{theorem}

Hall $\pi$-subgroups for a set $\pi$ of odd primes are relatively
common in simple groups of Lie type, but are much more restricted in
alternating and sporadic groups. We begin by treating the
alternating and sporadic cases.

\begin{lemma}\label{lem:noHall}
Let $n\in\mathbb{N}_{\geq 5}$ and let $\pi\subseteq [1,n]$ be a set
of prime numbers. Assume that $2\notin\pi$. Then the alternating
group $\Al_n$ admits a Hall $\pi$-subgroup if and only if $|\pi|=1$.
\end{lemma}

\begin{proof}
If $|\pi|=1$ the statement is obviously implied by the Sylow theory.
For the other implication, assume that $\Al_n$ admits a Hall
$\pi$-subgroup $H$ and suppose for a contradiction that $|\pi|\geq
2$. In particular let $\pi=\{p_1,p_2,\ldots, p_t\}$ for some odd
primes $p_1<p_2<\cdots<p_t$. 
Notice that $H$ is solvable, by the Feit--Thompson theorem
\cite{Thompson63}. By \cite{Hall1} we then know that $H$ admits a
$\{p_1,p_2\}$-Hall subgroup $K$. It follows that $K$ is a solvable
$\{p_1,p_2\}$-Hall subgroup of $\Al_n$. Moreover, since $2\notin
\{p_1,p_2\}$ we also have that $K$ is a solvable $\{p_1,p_2\}$-Hall
subgroup of the symmetric group $\Sy_n$. Using \cite[Theorem
A4]{Hall2}, we deduce that $|K|$ is even, and therefore $2\in
\{p_1,p_2\}\subseteq \pi$. This clearly contradicts our hypothesis.
%
%
\end{proof}

We note that the use of the Feit--Thompson theorem can be avoided by
instead relying on the known classification of non-solvable Hall
subgroups of symmetric groups. Suppose that $H$ is a non-solvable
Hall $\pi$-subgroup of $\Al_n$. Since $2 \notin \pi$, it follows
that $H$ is also a non-solvable Hall $\pi$-subgroup of $\Sy_n$. This
family of subgroups of symmetric groups is completely described in
\cite{Thompson}. In particular, either $H = \Sy_n$ or $n$ is a prime
number and $H = \Sy_{n-1}$. In both cases, we would have $2 \in
\pi$, which again contradicts our hypothesis.


\begin{proposition}\label{prop:alternatingandsporadic}
Theorem~\ref{thm:simple} holds when $S$ is an alternating group, a
sporadic group, or the Tits group ${}^2F_4(2)'$.
\end{proposition}

\begin{proof}
The case of alternating groups follows from Lemma~\ref{lem:noHall}
and Theorem~\ref{thm:Navarro-Tiep}. We now consider the sporadic
groups and the Tits group. By \cite[Theorem~6.14]{G86}, if $S$ has a
Hall $\pi$-subgroup for a set $\pi$ of odd primes, then $|\pi| \le
2$. The result then follows from \cite[Theorem~2.1]{GHSV21}, except
possibly in the cases $(S,\pi)=({}^2F_4(2)',\{3,5\})$, $(J_4,
\{23,43\})$, or $(J_4,\{29,43\})$. However, a direct check using
\cite{Atl1} shows that in each of these cases, $S$ does not have a
Hall $\pi$-subgroup.
\end{proof}


\begin{remark}
The hypothesis $2\notin \pi$ in
Proposition~\ref{prop:alternatingandsporadic} can not be removed.
Indeed, as already pointed out in the introduction, $\Al_7$
possesses a Hall $\{2,3,5\}$-subgroup but the only
$\{2,3,5\}'$-degree irreducible character of $\Al_7$ is the trivial
character. This observation can be extended to any prime larger than
$7$. Let $p\geq 11$ be a fixed prime number, and let $\pi$ be the
set consisting of all prime numbers strictly smaller than $p$. In
this case $\Al_{p-1}$ is a Hall $\pi$-subgroup of $\Al_p$ but
$\mathrm{Irr}_{\pi'}(\Al_p)=\{1_{\Al_p}\}$. In fact any non-trivial
irreducible character $\chi$ of $\pi'$-degree of $\Al_p$ would
satisfy $\chi(1)=p$. This is not possible because for any
$\zeta\in\mathrm{Irr}(\Sy_p)$ we have either $\zeta(1)\leq p-1$ or
$\zeta(1)\geq \frac{p(p-3)}{2}$, by \cite{Rasala}. Since $p-3> 4$ we
have that for any $\theta\in\mathrm{Irr}(\Al_p)$, either
$\theta(1)\leq p-1$ or $\theta(1)>p$.
\end{remark}

\smallskip

We now turn to the proof of Theorem~\ref{thm:simple} for simple
groups of Lie type $S \neq {}^2F_4(2)'$. By this, we mean simple
groups of the form $S = G/\bZ(G)$, where $G := \bG^F$ is the group
of fixed points of a simple, simply connected algebraic group $\bG$
defined over an algebraically closed field of characteristic $\ell$,
under a Steinberg endomorphism $F$ of $\bG$.

The case $\ell\notin\pi$ is easy.

\begin{proposition}\label{prop:ell-notin-pi}
Theorem \ref{thm:simple} holds when $S$ is a simple group of Lie
type in characteristic $\ell$ and $\ell\notin\pi$.
\end{proposition}

\begin{proof}
The Steinberg character $\St_G$ of $G$ is trivial on $\bZ(G)$ and
thus can be viewed as a character of $S$. Moreover, $\St_G$ is
rational-valued and has degree a power of $\ell$; see
\cite[Proposition~3.4.10]{GeckMalle}.
\end{proof}

We therefore focus on the case $\ell\in\pi$. Let $\bB$ be an
$F$-stable Borel subgroup of $\bG$ and $\bT$ a (maximally-split)
$F$-stable maximal torus of $\bG$ inside $\bB$. Let $\Phi$ be the
root system of $\bG$ with respect to $\bT$ and $\bB$ and $\Phi^+$
the set of corresponding positive roots. Let $\bU$ be the product of
the root subgroups corresponding to the roots in $\Phi^+$. This
$\bU$ is indeed the unipotent radical of $\bB$. We have
$U:=\bU^F\in\Syl_\ell(G)$ and $B=UT$, where $U:=\bU^F$ and
$T:=\bT^F$; see \cite[\S1.9 and \S2.9]{Carter85}.

Let $(\bG^*,F^*)$ be the pair dual to $(\bG,F)$, and set
$G^*:={\bG^*}^{F^*}$. Note that $S=[G^*,G^*]$, as $S$ is simple. Let
$\bT^*$ be an $F^*$-stable maximal torus of $\bG^*$ that is dual to
$\bT$ in the sense of \cite[Proposition~4.3.1]{Carter85}. Let
$\bB^*$ be an $F^*$-stable Borel subgroup of $\bG^*$ containing
$\bT^*$. Write $T^*:={\bT^*}^{F^*}$ and $B^*:={\bB^*}^{F^*}$. If
$\bU^*$ denotes the unipotent radical of $\bB^*$ and
$U^*:={\bU^*}^{F^*}$, then $B^*=U^*T^*$.

\begin{lemma}\label{lem:Gross}
Let $S$ be a simple group of Lie type in characteristic $\ell$.
Suppose that $\pi$ is a set of odd primes containing $\ell$ (in
particular, $\ell$ is odd) and that $S$ has a Hall $\pi$-subgroup.
Then $B^*$ contains a Hall $\pi$-subgroup of $G^*$.
\end{lemma}

\begin{proof}
Note that $\bZ(G)\le T\le B$. By \cite[Theorem~3.2]{G86}, the
quotient $B/\bZ(G)$, which may be regarded as a Borel subgroup of
$S$, contains a Hall $\pi$-subgroup of $S$. Equivalently, $|G:B| =
|S : B/\bZ(G)|$ is a $\pi'$-number. By Corollary~4.4.2 and
Proposition~4.4.4 of \cite{Carter85}, we have $|G^*|=|G|$ and
$|T^*|=|T|$, and it follows that $|U^*|=|U|$ and $|B|=|B^*|$. (Note
that $U$ and $U^*$ have orders equal to the $\ell$-parts of $|G|$
and $|G^*|$, respectively.) Now we have that $|G^*:B^*|$ is a
$\pi'$-number. As $B^*$ is solvable, it therefore contains a Hall
$\pi$-subgroup of $G^*$.
\end{proof}

\begin{lemma}\label{lem:Gross2}
Let $S\neq {}^2F_4(2)'$ be a simple group of Lie type in
characteristic $\ell$. To prove Theorem~\ref{thm:simple} for $S$, it
is sufficient to assume that $\pi\cap\pi(T^*)\neq\emptyset$.
\end{lemma}

\begin{proof}
By Proposition~\ref{prop:ell-notin-pi}, we may assume that
$\ell\in\pi$. If $\pi\cap \pi(S)=\{\ell\}$, then
Theorem~\ref{thm:simple} follows from
Theorem~\ref{thm:Navarro-Tiep}. We therefore may assume that
$|\pi\cap \pi(S)|\geq 2$. Since $S$ has a Hall $\pi$-subgroup, by
Lemma~\ref{lem:Gross}, the Borel subgroup $B^*$ contains a Hall
$\pi$-subgroup of $G^*$. In particular, $|\pi\cap \pi(B^*)|\geq 2$,
and hence there exists at least one prime lying in both $\pi$ and
$\pi(T^*)$.
\end{proof}

Recall that $S=G/\bZ(G)$, so the irreducible characters of $S$ are
precisely those irreducible characters of $G$ whose kernel contains
$\bZ(G)$. The set $\Irr(G)$ admits a natural partition into Lusztig
series $\EC(G,s)$ indexed by $G^\ast$-conjugacy classes of
semisimple elements $s\in G^\ast$. The series $\EC(G,s)$ consists of
those irreducible characters of $G$ that are constituents of some
Deligne--Lusztig character $R_{\bS}^{\bG}(\theta)$, where $\bS$ is
an $F$-stable maximal torus of $\bG$ and $\theta\in\Irr(\bS^F)$ is
such that the geometric conjugacy class of the pair $(\bS,\theta)$
corresponds to the $G^\ast$-conjugacy class containing $s$. See
\cite[\S2.6]{GeckMalle}.

We will look for the desired character among the so-called
\emph{semisimple characters}. Recall that \(\rho \in \Irr(G)\) is
called a semisimple character if the average value of \(\rho\) on
\(\mathfrak{C}^F\) is nonzero, where \(\mathfrak{C}\) is the
conjugacy class of \(\bG\) consisting of regular unipotent elements.
For each semisimple element \(s \in G^*\), the Lusztig series
\(\EC(G,s)\) contains one or more semisimple characters, all of
which have degree
\[
|G^* : \bC_{G^*}(s)|_{\ell'}.
\]
In fact, when \(\bC_{\bG^*}(s)\) is connected, \(\EC(G,s)\) contains
exactly one semisimple character; see \cite[p.~171]{GeckMalle}. We
record a well-known fact about these characters below. Here we use
$\ord(g)$ to denote the order of a group element $g$.

\begin{lemma}\label{lem:1}
Assume the above notation. Assume furthermore that $\bC_{\bG^*}(s)$
is connected. Then the unique semisimple character, say $\chi_s$, in
the series \(\EC(G,s)\) is trivial on $\bZ(G)$ if and only if $s\in
S$. In such situation, $\QQ(\chi_s)\subseteq \QQ(\zeta_{\ord(s)})$.
\end{lemma}

\begin{proof}
The first part follows from, for instance,
\cite[Proposition~24.21]{malletesterman}, \cite[Lemma~2.2]{M07}, and
\cite[Lemma~5.8]{Hung22}. The latter part is
\cite[Lemma~4.2]{GHSV21}.
\end{proof}

We handle the linear and unitary groups separately. To unify
notation, we write $S=\PSL_n^\epsilon(q)$, where the superscript
$\epsilon=+1$ corresponds to the linear groups and $\epsilon=-1$ to
the unitary groups. We use analogous notation for the related
groups, for example $G=\SL_n^\epsilon(q)$ and
$G^*=\PGL_n^\epsilon(q)$. It is more convenient to first study the
characters of $\widetilde{G}:=\GL_n^\epsilon(q)$ and then analyze
those of $G=\SL_n^\epsilon(q)$ as irreducible constituents of
restricted characters. Note that $\widetilde{G}$ is self-dual, and
we will identify $\widetilde{G}$ with its dual group. Let
$\pi:\widetilde{G}\to G^*$ be the natural projection from
$\widetilde{G}$ to $G^*$.

Let $\widetilde{s}$ be a semisimple element of $\widetilde{G}$.
Since the ambient algebraic group of $\widetilde{G}$ has connected
center, the Lusztig series
$\mathcal{E}(\widetilde{G},\widetilde{s})$ contains a unique
semisimple character, which we denote by $\chi_{\widetilde{s}}$.

\begin{lemma}\label{lem:numberofconstituents}
The number of irreducible constituents of the restriction of
semisimple character $\chi_{\widetilde{s}}$ to $G$ divides
$\gcd(\ord(\widetilde{s}),n,q-\epsilon)$.
\end{lemma}

\begin{proof}
Set $s:=\pi(\widetilde{s})$. Recall that
$\chi_{\widetilde{s}}(1)=|\widetilde{G}:\bC_{\widetilde{G}}(\widetilde{s})|_{\ell'}$.
The restriction of $\chi_{\widetilde{s}}$ from $\widetilde{G}$ to
$G$ is multiplicity-free and its irreducible constituents are
precisely the semisimple characters of the Lusztig series
\(\EC(G,s)\) by \cite[Corollary~2.6.18]{GeckMalle}. Each of these
constituent has degree $|G^* : \bC_{G^*}(s)|_{\ell'}$. Therefore,
the number of irreducible constituents of the restriction is
\[
\frac{|\widetilde{G}:\bC_{\widetilde{G}}(\widetilde{s})|_{\ell'}}{|G^*
:
\bC_{G^*}(s)|_{\ell'}}=\frac{|\widetilde{G}:\bC_{\widetilde{G}}(\widetilde{s})|_{\ell'}}{|\widetilde{G}
:
\pi^{-1}(\bC_{G^*}(s))|_{\ell'}}=|\pi^{-1}(\bC_{G^*}(s)):\bC_{\widetilde{G}}(\widetilde{s})|_{\ell'}.
\]

Let $Z$ be the (cyclic) subgroup of the multiplicative group
$\FF_{q^2-1}^\times$ of order $q-\epsilon$. Consider the mapping
$\kappa: \pi^{-1}(\bC_{G^*}(s))\to Z$ defined by
$g\widetilde{s}g^{-1}=\kappa(g)\widetilde{s}$. This is indeed a
homomorphism with $\Ker(\kappa)=\bC_{\widetilde{G}}(\widetilde{s})$.
It follows that the index
$|\pi^{-1}(\bC_{G^*}(s)):\bC_{\widetilde{G}}(\widetilde{s})|$ is
equal to the order of the image $\Im(\kappa)$ of $\kappa$.
Therefore, the lemma follows once we show that
\[
|\Im(\kappa)| \text{ divides }
\gcd(\ord(\widetilde{s}),n,q-\epsilon).
\]

First, it is clear that $|\Im(\kappa)|\mid (q-\epsilon)$. Now let
$g$ be an arbitrary element of $\pi^{-1}(\bC_{G^*}(s))$. We have
$\det(\widetilde{s})=\det(g\widetilde{s}g^{-1})=\det(\kappa(g)\widetilde{s})=\kappa(g)^n\det(\widetilde{s})$,
and hence $\kappa(g)^n=1$; equivalently, $\ord(\kappa(g))\mid n$.
Finally, we observe that
$\ord(\widetilde{s})=\ord(g\widetilde{s}g^{-1})=\ord(\kappa(g)\widetilde{s})=\lcm(\ord(\kappa(g)),\widetilde{s})$,
which implies that $\ord(\kappa(g))\mid \ord(\widetilde{s})$.
\end{proof}

The proof of Lemma~\ref{lem:numberofconstituents} also yeilds the
following well-known result; however, we only require the first
statement.

\begin{lemma}\label{lem:2}
Let $s := \pi(\widetilde{s}) \in G^*$, where $\widetilde{s}$ is a
semisimple element of $\widetilde{G}$. Suppose that the multiset of
eigenvalues of $\widetilde{s}$ is not invariant under multiplication
by every nontrivial root of unity. Then the centralizer
$\bC_{\bG^*}(s)$ is connected. Moreover, if the multiset of
eigenvalues of $\widetilde{s}$ is not invariant under multiplication
by every nontrivial element of the subgroup of $\FF_{q^2-1}^\times$
of order $q-\epsilon$, then $\chi_{\widetilde{s}}$ restricts
irreducibly to $G$.
\end{lemma}

\begin{proof}
For every $g \in \pi^{-1}\bigl(\bC_{G^*}(s)\bigr)$, we have
$g\widetilde{s}g^{-1} = \kappa(g)\widetilde{s}$, and hence
$g\widetilde{s}g^{-1}$ and $\widetilde{s}$ have the same
eigenvalues. Therefore, if the multiset of eigenvalues of
$\widetilde{s}$ is not invariant under multiplication by every
nontrivial element of the subgroup of $\FF_{q^2-1}^\times$ of order
$q-\epsilon$, the homomorphism $\kappa$ must have trivial image,
which proves the second statement.

A similar argument, now in the setting of algebraic groups
$\widetilde{\bG} := \GL^\epsilon_n(\overline{\FF}_{\ell})$ and
${\bG^*} := \PGL^\epsilon_n(\overline{\FF}_{\ell})$ instead, shows
that if the multiset of eigenvalues of $\widetilde{s}$ is not
invariant under multiplication by every nontrivial root of unity,
then $\bC_{\bG^*}(s)$ is the image of the connected group
$\bC_{\widetilde{\bG}}(\widetilde{s})$ under $\pi$ and hence is
itself connected.
\end{proof}

\begin{proposition}\label{prop:linear-unitary}
Theorem~\ref{thm:simple} holds for the simple groups $S=\PSL_n(q)$
and $\PSU_n(q)$, where $n\geq 2$ and $q$ is a prime power.
\end{proposition}

\begin{proof}
Recall that $T^*$ is a maximal torus contained in a Borel subgroup
$B^*$ of $G^*$. By Lemma~\ref{lem:Gross2}, we may assume that $\pi$
and $\pi(T^*)$ share a common prime, say $p$. In the linear case, we
have
\[
|T^*| = (q-1)^{n-1},
\]
so $p$ divides $q-1$. In the unitary case,
\[
|T^*| =
\begin{cases}
(q^2-1)^{(n-1)/2}, & \text{if $n$ is odd},\\[2pt]
(q^2-1)^{(n-2)/2}(q-1), & \text{if $n$ is even},
\end{cases}
\]
and hence $p$ divides $q^2-1$ (see \cite[p.~74]{Carter85}). Since
$B^*$ contains a Hall $\pi$-subgroup of $G^*$ by
Lemma~\ref{lem:Gross}, a comparison of the orders of $G^*$ and $B^*$
shows that $p$ cannot divide $q+1$. Therefore, as in the linear
case, we again conclude that $p$ divides $q-1$.

We first consider the case $\epsilon=1$. Let $\delta$ be a generator
of $\FF_{q}^\times$, and define the semisimple element
\[
\widetilde{s} := \diag\!\bigl( \delta^{(q-1)/p},
\delta^{(q-1)(p-1)/p}, 1^{\,n-2} \bigr) \in \widetilde{G}.
\]
Let $s$ denote the image $\pi(\widetilde{s})$ of $\widetilde{s}$
under $\pi:\widetilde{G}\to G^*$. Then
\[
\ord(s) = \ord(\widetilde{s}) = p.
\]
Taking conjugation if necessary, we may assume that $s\in T^*$.
Moreover, as $\det(\widetilde{s})=\delta^{q-1}=1$, we have
\[
\widetilde{s}\in G \quad\text{and}\quad s\in S\cap T^*.
\]

Note that the case $(n,p)=(3,3)$ cannot occur. (Otherwise, by
Lemma~\ref{lem:Gross}, the Borel subgroup $B^*$, of order
$q^3(q-1)^2$, would contain a Sylow $3$-subgroup of $G^*$ whose
order $q^3(q^2-1)(q^3-1)$ has larger $3$-part, leading to a
contradiction.) It is then easy to check that the multiset of
eigenvalues of $\widetilde{s}$ is not invariant under multiplication
by every nontrivial root of unity. By Lemma~\ref{lem:2}, it follows
that $\bC_{\bG^*}(s)$ is connected, and we let $\chi_s$ be the
unique semisimple character in the series $\EC(G,s)$. Using
Lemma~\ref{lem:1}, we deduce that
\[\QQ(\chi_{s})\subseteq\QQ(\zeta_{\ord(s)})=\QQ(\zeta_p).\] Moreover,
since $s\in T^*$ and $T^*$ is abelian, we have \[T^*\le
\bC_{\bG^*}(s).\] It follows that \[|G^*:\bC_{\bG^*}(s)| \text{
divides } |G^*:T^*|,\] and hence
\[
\chi_s(1)=|G^*:\bC_{G^*}(s)|_{\ell'} \text{ divides }
|G^*:T^*|_{\ell'}.
\]
As $|G^*:T^*|_{\ell'}=|G^*:B^*|$ is a $\pi'$-number by
Lemma~\ref{lem:Gross}, we conclude that $\chi_s(1)$ is also a
$\pi'$-number. Also, from the fact that $s\in S=[G^*,G^*]$, we know
that $\chi_{s}$ is trivial on $\bZ({G})$. This completes the proof
for linear groups.

It remains to consider the case $\epsilon=-1$ and $p \mid (q-1)$.
Let \(\widetilde{t}\) be a semisimple element of order \(p\) in a
maximal torus of order $q^2-1$ of \(\GU_2(q)\). Since \(p \mid
(q-1)\) and \(p\) is odd, we have \(\gcd(p,q+1)=1\), and hence
\(\widetilde{t}\in \SU_2(q)\). Define the semisimple element
\[
\widetilde{s}:=\diag\left(\widetilde{t}, I_{n-2}\right)\in
\widetilde{G},
\]
and consider the corresponding semisimple character
\(\chi_{\widetilde{s}}\in \Irr(\widetilde{G})\). Note that
$\widetilde{s}$ belongs to a maximal torus of $\widetilde{G}$ of
order $(q^2-1)^{(n-1)/2}(q+1)$ if $n$ is odd, and of order
$(q^2-1)^{n/2}$ if $n$ is even.

Note also that \(\ord(\widetilde{s})=\ord(\widetilde{t})=p\), which
does not divide $q+1$. Lemma~\ref{lem:numberofconstituents} then
implies that \(\chi_{\widetilde{s}}\) restricts irreducibly to
\(G\). Arguments similar to those in the linear case show that this
restriction is trivial on \(\bZ(G)\), has \(\pi'\)-degree, and takes
values in \(\QQ(\zeta_p)\). This completes the proof.
\end{proof}

We can now finish the proof of Theorem~\ref{thm:simple}.

\begin{proof}[Proof of Theorem~\ref{thm:simple}]
By Propositions~\ref{prop:alternatingandsporadic} and
\ref{prop:linear-unitary}, and the classification of finite simple
groups, we may assume that $S \neq {}^2F_4(2)'$ is a simple group of
Lie type not of type $A$. By Lemma~\ref{lem:Gross2}, there exists at
least one prime lying in both $\pi$ and $\pi(T^*)$. As before, let
$p$ be such a prime.

Note that
\[
|T^*:(T^*\cap S)| = |T^*S:S| \text{ divides } |G^*:S|
\]
and $|G^*:S|$ is the order of the group of \emph{diagonal
automorphisms} of $S$. On the other hand, the order of $T^*$ is
given by
\[
|T^*|=\prod_{\mathfrak{o}\in \mathcal{O}}
\bigl(q^{|\mathfrak{o}|}-1\bigr),
\]
where $\mathcal{O}$ is the set of $F^*$-orbits on the simple roots
of the root system of $\bG^*$ associated with $\bB^*$ and $\bT^*$,
and $q$ is the absolute value of the eigenvalues of $F^*$ on the
character group of $\bT^*$; see \cite[p.~74]{Carter85}. A
straightforward case-by-case check, using \cite[Table~5]{Atl1} for
the size of the group of diagonal automorphisms and
\cite[\S1.19]{Carter85} for the sizes of the $F^*$-orbits, reveals
that if an odd prime $p$ is a divisor of $|G^*:S|$, then indeed
$(|G^*:S|_p)^2$ divides $|T^*|$, and it follows that $p$ also
divides $|T^*\cap S|$. (The only exception is the case $S=\PSU_3(q)$
with $(q+1)_3=3$, but this was already excluded using
Proposition~\ref{prop:linear-unitary}.) Therefore, we may and will
assume from now on that \[p\in \pi\cap \pi(T^*\cap S).\]
Consequently, there exists a semisimple element $s\in G^*$ such that
\[
s\in T^*\cap S \quad\text{and}\quad \ord(s)=p.
\]

Suppose first that $p\nmid |\bZ(\bG)|$. Then $\bC_{\bG^*}(s)$ is
connected by \cite[Corollary~4.6]{Borel70}. Arguing as in the proof
of Proposition~\ref{prop:linear-unitary} and applying
Lemmas~\ref{lem:Gross} and \ref{lem:1}, we have that the unique
semisimple character $\chi_s$ in $\EC(G,s)$ has $\pi'$-degree, is
trivial on $\bZ(G)$, and satisfies $\QQ(\chi_s)\subseteq
\QQ(\zeta_p)$.

Next suppose that $p\mid |\bZ(\bG)|$. By
\cite[Theorem~1.12.5]{GLS98} and the assumption that $p$ is odd, we
are left with only the case $p=3$ and $\bG$ to be of type $E_6$.
Here, $G=E_6^\epsilon(q)$ with $\epsilon\in\{\pm1\}$, where
$\epsilon=1$ corresponds to the untwisted groups and $\epsilon=-1$
to the twisted groups, and $|\bZ(G)|=\gcd(3,q-\epsilon)=3$.

As mentioned above, all the semisimple characters in the Lusztig
series $\EC(G,s)$ indexed by $s$ have the same degree
$|G^*:\bC_{G^*}(s)|_{\ell'}$. This is a $\pi'$-number due to the
fact that $s\in T^*$, as argued in the proof of
Proposition~\ref{prop:linear-unitary}. Moreover, as noted in the
proof of \cite[Proposition~4.5]{GHSV21}, these semisimple characters
take integer values on unipotent elements and therefore have field
of values contained in $\QQ(\zeta_3)$, by \cite[Lemma~4.3]{GHSV21}.
Finally, since their degrees are coprime to $3$, their restrictions
to $\bZ(G)$ are multiples of the trivial character. (Let $\chi$ be
such a character and let $\chi_{\bZ(G)}=\chi(1)\alpha$ for some
$\alpha\in\Irr(\bZ(G))$. Then
$\mathbf{1}_{\bZ(G)}=\det(\chi)_{\bZ(G)}=\det(\chi_{\bZ(G)})=\det(\chi(1)\alpha)=\alpha^{\chi(1)}$,
which implies that the order of $\alpha$ divides
$(\chi(1),|\bZ(G)|)$.) This completes the proof.
\end{proof}

\section{A Proof of Theorem \ref{thm:main}}

We are now in the position to give a complete proof of our main result, as stated in the introduction.

A finite group admitting a Hall $\pi$-subgroup is often called an
\emph{$E_\pi$-group}. There is a large literature on the theory of
$E_\pi$-groups, including the results about simple groups we cited in
the previous section. We refer the reader to \cite{RV10} for the
latest results and relevant information. Note that the class of
$E_\pi$-groups is not closed under extensions and a subgroup of an
$E_\pi$-group might be no longer an $E_\pi$-group. Nevertheless, the
only fact we need is that every normal/subnormal subgroup and every
quotient of an $E_\pi$-group is an $E_\pi$-group. In fact, if $H$ is
a Hall $\pi$-subgroup of $G$ and $N\nor G$, then $H\cap N$ is a Hall
$\pi$-subgroup of $N$ and $HN/N$ is a Hall $\pi$-subgroup of $G/N$.

Another fact we need is that if $N\nor G$ such that $|G:N|=r$ is an
odd prime and $\theta\in \Irr(N)$ with $\QQ(\theta)\subseteq
\QQ(\zeta_p)$ for some prime $p\neq r$, then there exists an
irreducible constituent $\chi$ of $\theta^G$ such that
$\QQ(\chi)\subseteq \QQ(\zeta_p)$. This follows from
\cite[Lemma~2.2]{GHSV21}, for instance.

We restate Theorem~\ref{thm:main} for the reader's convenience.

\begin{theorem}
Let $\pi$ be a set of odd primes and $G$ a finite group such that
$G$ has a nontrivial Hall $\pi$-subgroup. Then $G$ possesses a
nontrivial $\pi'$-degree irreducible character with values in
$\QQ(\zeta_p)$ for some $p\in\pi$.
\end{theorem}

\begin{proof}
We argue by induction on $|G|$. Let $M \triangleleft\, G$ be a
normal subgroup such that $G/M$ is simple.

First assume that $G/M$ is abelian and that $|G/M| \in \{2\} \cup
\pi$. Then the inflation to $G$ of any nontrivial linear character
of $G/M$ satisfies the required conditions. Now suppose that $r :=
|G/M|$ is an odd prime not belonging to $\pi$. In particular, we
have $\pi(M) \supseteq \pi\cap\pi(G)$.

By the induction hypothesis and the fact that $M$ also has a Hall
$\pi$-subgroup, there exists $\psi \in \Irr_{\pi'}(M)$ and some $p
\in \pi$ such that $\QQ(\psi) \subseteq \QQ(\zeta_p)$. Let $\chi \in
\Irr(G)$ lie over $\psi$. By~\cite[Corollary~6.19]{Isaacs1}, either
$\chi_M = \psi$ or $\chi_M = \sum_{i=1}^r \psi_i$ is the sum of the
$G$-conjugates of $\psi$.

In the latter case, we may take $\chi = \psi^G$, with a note that
$\chi(1) = r \psi(1)$ is a $\pi'$-number and $\QQ(\chi) \subseteq
\QQ(\psi)\subseteq \QQ(\zeta_p)$. In the former case, every
irreducible character of $G$ lying over $\psi$ is an extension of
$\psi$, and hence has $\pi'$-degree. Moreover, as mentioned above,
one of these extensions has values in $\QQ(\zeta_p)$, as required.

Finally, suppose that $G/M$ is non-abelian. Set $\pi_1 := \pi(G/M)
\cap \pi$. (Note that $\pi_1$ could be empty.) Again, the group
$G/M$ has a Hall $\pi$-subgroup, which is also a Hall
$\pi_1$-subgroup. Theorem~\ref{thm:simple} then yields a character
$\chi \in \Irr_{\pi_1'}(G/M)$ such that $\QQ(\chi) \subseteq
\QQ(\zeta_p)$ for some $p \in \pi_1$. Since $\chi(1)$ divides
$|G/M|$, the character $\chi$ is also of $\pi'$-degree. The desired
character is obtained by inflating $\chi$ to $G$.
\end{proof}

We conclude with a remark that, in view of the above proof, the
conclusion of Theorem~\ref{thm:main} remains valid when $2 \in \pi$,
provided that the group $G$ is $\pi$-separable.


\end{document}